\documentclass[12pt]{amsart}

\usepackage{amsthm,amsmath,amssymb}
\newtheorem{thm}{Theorem} 
\newtheorem{cor}[thm]{Corollary}
\newtheorem{lm}[thm]{Lemma}

\newtheorem{df}[thm]{Definition}
\newtheorem{prob}[thm]{Problem}
\newcommand{\op}{\operatorname}

\newcommand{\meet}{\wedge}
\newcommand{\join}{\vee}
\newcommand{\Meet}{\bigwedge}
\renewcommand{\Join}{\bigvee}

\date{\today}
\begin{document}

\title{A Class of Infinite Convex Geometries}
%%\author {Kira \& JB}
\author{Kira Adaricheva}
\address{Department of Mathematical Sciences, Yeshiva University, New York, NY 10016 USA}
\email{adariche@yu.edu}
\address{and Department of Mathematics, School of Science and Technology, Nazarbayev University, Astana, Kazakhstan}
\email{kira.adaricheva@nu.edu.kz}
\author{J. B. Nation}
\address{Department of Mathematics, University of Hawaii, Honolulu, HI 96822, USA}
\email{jb@math.hawaii.edu}

\begin{abstract}
Various characterizations of finite convex geometries are well known.
This note provides similar characterizations for possibly infinite convex 
geometries whose lattice of closed sets is strongly coatomic and spatial.
Some classes of examples of such convex geometries are given.
%%e-jc

\bigskip\noindent \textbf{Keywords:} convex geometry; closure system; anti-exchange property; join semidistributive lattice; strongly coatomic and spatial lattice
\end{abstract}

\maketitle

\section{Introduction}

There are various ways to characterize finite convex geometries;
see Chapter 3 of our \cite{AN15}, which combines results from 
Dilworth \cite{D2},
Avann \cite{Av61},
Edelman and Jamison \cite{EJ85},
Duquenne \cite{D91}, and
Monjardet \cite{Mo85}.
These characterizations can be either combinatorial or lattice theoretical
in nature.

Infinite versions of convex geometries occur in several sources, each of which 
is equivalent to a closure operator with the anti-exchange property, plus some 
finiteness conditions to provide structure. 
Crawley and Dilworth \cite{CD60} consider dually algebraic, strongly coatomic, 
locally distributive lattices.  
Semenova \cite{MS00, MS00B} extends these results to lower continuous, strongly coatomic, 
locally distributive lattices.
Adaricheva, Gorbunov and Tumanov \cite{AGT03} discuss closure operators with the 
anti-exchange property 
whose closure lattices are weakly atomic, dually spatial and atomistic.
Adaricheva and Nation \cite{AN15} are concerned with 
algebraic closure operators with the anti-exchange property.
See also Sakaki \cite{S15}, 
Adaricheva and Pouzet \cite{AP11}, and Adaricheva \cite{A14}.   

Here we consider a class of complete lattices for which most of the 
equivalent characterizations of finite convex geometries remain valid,
\emph{viz.}, lattices that are strongly coatomic and spatial.
(A complete lattice is \emph{spatial} if every element is a join of 
completely join irreducible elements.)
Our main result is Theorem~\ref{SCSgeom}, giving various characterizations
of the convex geometries within this class.
Moreover, the class is rich enough to provide manifold examples of
infinite convex geometries.

\section{The anti-exchange property}

\begin{df}\label{def:AEP}
A closure system $(X, \gamma)$ satisfies the \emph{anti-exchange property} if for all 
$x \neq y$ and all closed sets $A \subseteq X$,
 \begin{equation}
 x\in\gamma(A\cup\{y\})\text{ and }x\notin A \text{ imply that }y\notin\gamma(A\cup\{x\}).
 \tag{AEP}
 \end{equation}
\end{df}
Equivalently, a closure operator satisfies the anti-exchange property if for all closed sets 
$A \subseteq X$ and elements $x$, $y \notin A$, if 
$\gamma(A \cup \{ x \}) = \gamma(A \cup \{ y \})$ then $x=y$.

Examples of closure operators with the anti-exchange property include
\begin{itemize}
\item the convex hull operator on Euclidean space $\mathbb E^n$,
\item the convex hull operator on an ordered set,
\item the subalgebra-generated-by operator on a semilattice,
\item the algebraic-subset-generated-by operator on a complete lattice.
\end{itemize}
For a closure system $(X,\gamma)$, we will let $\op{Cld}(X,\gamma)$ denote the 
lattice of $\gamma$-closed subsets of $X$.
A closure system is \emph{zero-closure} if $\gamma(\emptyset)=\emptyset$.

\begin{df}\label{D:WCG}
A zero-closure system that satisfies the anti-exchange property is called a 
\emph{convex geometry}.
\end{df}

(This common convention is a bit awkward, as some useful closure operators with the 
anti-exchange property have a non-empty closure of the empty set.  Nonetheless, we 
shall retain it.)

A lattice is \emph{strongly coatomic} if $a < c$ in $L$ implies that there exists
$b$ such that $a \leq b \prec c$.
A closure system is strongly coatomic if its lattice of closed sets is so.

\begin{thm}\label{covCG}
For a strongly coatomic closure system $(X,\gamma)$, the following are equivalent.
\begin{itemize}
\item[(1)] $(X,\gamma)$ has the anti-exchange property.
\item[(2)] If $A\prec B$ in $\op{Cld} (X,\gamma)$, then $|B\setminus A|=1$.
\end{itemize}
\end{thm}

\begin{proof}
Assume that $(X,\gamma)$ has the AEP.  If $A\prec B$ in $\op{Cld} (X,\gamma)$ 
and $x$, $y \in B\setminus A$, then $\gamma(A \cup \{ x \}) = B = 
\gamma(A \cup \{ y \})$, whence $x=y$ by the AEP.

Suppose that $(X,\gamma)$ satisfies (2).  Assume that
$B = \gamma(A \cup \{ x \}) = \gamma(A \cup \{ y \}) > A = \gamma(A)$. 
As $\op{Cld}(X,\gamma)$ is strongly coatomic, there is a closed set $A'$ 
such that $A \leq A' \prec B$.  Then
$B = \gamma(A' \cup \{ x \}) = \gamma(A' \cup \{ y \})$, 
so $x$, $y \in B \setminus A'$.  By (2) we have $x=y$, as desired.
\end{proof}

The equivalence of the preceding theorem is also valid for algebraic closure systems; 
see \cite{A14} and \cite{AN15}.

\section{Strongly spatial lattices}

Our goal in this paper is to find a general class of lattices 
(not necessarily finite)  to which the characterizations of finite 
convex geometries extend naturally.  Recall that in a finite lattice,
every element is a join of join irreducible elements.

\begin{lm} \label{basicJI}
In a strongly coatomic complete lattice,
every nonzero join irreducible element is completely join irreducible.
\end{lm}

\begin{proof}
We prove the contrapositive.  Assume that $w = \Join X$ with $x < w$ for all $x \in X$.  
Choose any $x_0 \in X$.  
Then $x_0 < w$, so by strong coatomicity there exists $c \in L$ such that $x_0 \leq c \prec w$.  
Since $\Join X = w > c$, there exists an element $x_1 \in X$ such that $x_1 \nleq c$.
But then $w = x_1 \join c$ is finitely join reducible.  
\end{proof}

A complete lattice in which every element is a join of completely join irreducible elements
is said to be \emph{spatial}.  Clearly, this is a desirable property
for any sort of ``geometry."  
In \cite{AP11}, it was shown that every weakly atomic convex geometry is spatial. 

The set of nonzero join irreducible elements of a lattice $L$ will be denoted by $\op{Ji}(L)$.
The preceding lemma says that in dealing with strongly coatomic lattices, we need not
worry about the distinction between join irreducible and completely join irreducible
elements.

Note that every complete lattice can be represented \emph{via} a closure system in various ways.
If the lattice is spatial, then a standard representation
would use the completely join irreducible elements of $L$ as the set $X$.   
The next observation relates this to condition (2) of Theorem~\ref{covCG}.

\begin{lm} 
Let $(X,\gamma)$ be a closure system such that the lattice of closed subsets
$\op{Cld}(X,\gamma)$ is strongly coatomic and satisfies the property that $A \prec B$
implies $|B \setminus A|=1$.
Then there is a one-to-one correspondence betweeen
$X \setminus \gamma(\emptyset)$ and the nonzero completely join irreducible closed sets of\/ $(X,\gamma)$.
\end{lm}

\begin{proof}
Always in a closure system, if $B$ is a completely join irreducible closed set, 
then $B = \gamma(x)$ for any 
$x \in B \setminus B_*$, where $B_*$ denotes the unique lower cover of $B$.  
On the other hand, if $\op{Cld}(X,\gamma)$ is strongly coatomic and 
$x \notin \gamma(\emptyset)$, then there is a closed set $A$ such that $A \prec \gamma(x)$. 
If $(X,\gamma)$ also satisfies the condition, so that 
$|\gamma(x) \setminus A| = 1$, then the only choice for $A$ is 
$\gamma(x) \setminus \{ x \}$.
That makes $\gamma(x) \setminus \{ x \}$ the unique lower cover of $\gamma(x)$, 
so that $\gamma(x)$ is completely join irreducible.
\end{proof}

It is useful to identify a stronger version of spatiality.  
A complete lattice $L$ is said to be \emph{strongly spatial} if whenever $a \nleq b$
in $L$, then there exists an element $p$ that is minimal with respect to
the property that $x \leq a$, $x \nleq b$.  
Clearly any such element must be completely join irreducible,
and thus a strongly spatial lattice is spatial.

Consider the lattice $K$ which is a union of two chains, an infinite chain
$0 < \dots < a_2 < a_1 < 1$ and a 3-element chain $0<b<1$, with no other relations
holding.  Then $K$ is strongly coatomic and spatial, but not strongly spatial.

In the remainder of this section, we identify some conditions under which strongly
coatomic, spatial lattices will be strongly spatial.
The first one is easy.  

A complete lattice $L$ is called \emph{atomistic} if every nonzero $a\in L$ is a join of atoms. 

\begin{lm}
Every atomistic lattice is strongly spatial.
\end{lm}

A lattice $L$ is called \emph{lower semimodular} if 
$a \prec b$ implies $a \wedge c \preceq b \wedge c$ for all $a$, $b$, $c \in L$.  
Equivalently, a lattice is lower semimodular if $a \prec a \join c$ implies 
$a \wedge c \prec  c$.  

\begin{lm} \label{lsmtoss}
If a complete lattice is strongly coatomic, lower semimodular and spatial, then it
is strongly spatial.
\end{lm}

\begin{proof}
Let $a \nleq b$ in $L$.  Then $a > a \meet b$, and hence there exists $c$ such that
$a \succ c \geq a \meet b$.  As $L$ is spatial, there is a join irreducible element
$p$ with $p \leq a$, $p \nleq c$.  By lower semimodularity, $p \succ p \meet c$, and
thus $p$ is a minmal element with this property.
\end{proof}

A subset $D$ of a lattice $L$ is \emph{down-directed} if for every pair $d_1$, $d_2 \in D$ 
there exists $d_3 \in D$ such that $d_1 \geq d_3$ and $d_2 \geq d_3$ both hold.
(Note that every chain in $L$ is a down-directed set.)
A complete lattice is \emph{lower continuous} if 
$a \join \Meet D = \Meet_{d \in D} (a \join d)$  
for every down-directed set $D \subseteq L$.
A standard result of lattice theory is
that every dually algebraic lattice is lower continuous.

\begin{lm}
If a complete lattice is strongly coatomic and lower continuous, then it
is strongly spatial.
\end{lm}

\begin{proof}
Again let $a \nleq b$ in $L$, so that $a > a \meet b$, and choose $c$ such that
$a \succ c \geq a \meet b$.  Let $\mathcal P = \{ x \in L : x \leq a, \ x \nleq c \}$,
and note that $x \in \mathcal P$ if and only if $c \join x = a$.   It follows from
lower continuity that if $D$ is a chain in $\mathcal P$, then $\Meet D \in \mathcal P$. 
By Zorn's Lemma, $\mathcal P$ contains a minimal element, which is the desired conclusion.
\end{proof}

\section{Strongly coatomic, strongly spatial, join semidistributive lattices}

Next we generalize some equivalences of join semidistributivity which are
well-known for finite lattices.

The implication
\[ (\op{SD}_{\join}) \qquad
w = x \join y = x \join z  \quad\text{implies}\quad w = x \join (y \meet z)
\]
is known as the \emph{join semidistributive law}.  
In view of the results in J\'onsson and Kiefer \cite{JK62},
we consider the following version of the law:
\[
(\op{SD}_{\join}^*) \qquad 
w = \Join Y = \Join Z \quad\text{implies}\quad w = \Join (y \meet z)
\]
where the sets $Y$, $Z$ are potentially infinite.

For subsets $A$, $B \subseteq L$ we say that $A$ \emph{refines} $B$, denoted $A \ll B$,
if for every $a \in A$ there exists $b \in B$ such that $a \leq b$.  Note that
$A \ll B$ implies $\Join A \leq \Join B$. 

We say that $w = \Join A$ is a \emph{canonical join decomposition} if the join
is irredundant, and $w = \Join B$ implies $A \ll B$.  This implies that
$A \subseteq \op{Ji}(L)$.

\begin{thm}\label{CharSD}
The following are equivalent in a strongly coatomic, strongly spatial lattice $L$.
\begin{enumerate}
\item $L$ satisfies $\op{SD}_\join$.
\item $L$ satisfies $\op{SD}_\join^*$.
\item Every element of $L$ has a canonical join decomposition.
\item If $w \succ c$ in $L$, then there exists a join irreducible $k$ 
which is the unique minimal element such that $k \leq w$ but $k \nleq c$.
\end{enumerate}
\end{thm}

\begin{proof}
To see that $(1) \Rightarrow (4)$, note that the strongly spatial property says
that there is at least one such element $k$.  If there were two or more, say $k_1$ and $k_2$,
then $c \join k_1 = w= c \join k_2 > c \join (k_1 \meet k_2)$ by the minimality of each $k_i$, 
contradicting $(\op{SD}_{\join})$.  Hence such an element $k$ is unique.

$(4) \Rightarrow (3)$.
Assume that property (4) holds, and fix an element $w \in L$.  
Let $C = \{ c \in L : w \succ c \}$ be the set of lower covers of $w$.  
For each $c \in C$, we can find a minimal element
$k_c$ such that $k_c \leq w$ but $k_c \nleq c$.  We claim that $w = \Join_{c \in C} k_c$
canonically.  Clearly $\Join_{c \in C} k_c=w$, since each $k_c$ is below $w$,
while the join is below no lower cover of $w$.  

If $c \ne d \in C$, then $d \nleq c$.  By strong spatiality, there exists
a minimal element $p$ such that $p \leq d$, $p \nleq c$.  
As $k_c$ is assumed to be the unique minimal element below $w$ but not $c$, 
we must have $p=k_c$, whence $k_c \leq d$.  
Thus $\Join_{c \ne d} k_c \leq d$.
It follows that the join representation
$w = \Join_{c \in C} k_c$ is irredundant.

Suppose $w = \Join A$ and consider $c \prec w$.  There exists some $a_0 \in A$
such that $a_0 \nleq c$, though $a_0 \leq w$, whereupon $k_c \leq a_0$.
Since this holds for all $c \in C$, we have $\{ k_c : c \in C \} \ll A$, as desired.

$(3) \Rightarrow (2)$.
Suppose that $w = \Join S = \Join T$.  If there is a canonical join decomposition
$w = \Join U$ in $L$, then $U$ refines both $S$ and $T$, so that for each
$u \in U$ there exist $s \in S$ with $u \leq s$, 
and $t \in T$ with $u \leq t$.  Hence each $u \leq s \meet t$
for some pair, and it follows that $w = \Join (s \meet t)$.

$(2) \Rightarrow (1)$ clearly, as $\op{SD}_\join$ is a special case of 
$\op{SD}_\join^*$.
\end{proof}

\section{Characterizations of strongly coatomic, spatial convex geometries}

First, we introduce an idea with a geometric flavor.

For any $A\subseteq X$, $x \in A$ is called an \emph{extreme point} of $A$ if
$x \notin \gamma(A\setminus \{x\})$. The set of extreme points of $A$ is denoted 
$\op{Ex}(A)$.
In lattice terms, for a strongly spatial lattice $L$, we identify the element $a$ with the set 
$\op{Ji}(a) = \{ p \in \op{Ji}(L) : p \leq a \}$.
Then $x \in \op{Ji}(a)$ is an extreme point of $a$ if 
$a > \Join (\op{Ji}(a) \setminus \{ x \})$.
This means that (i) $x$ is completely join prime in the ideal $\op{id}(a)$, and 
(ii) there is no other join irreducible $y$ with $x < y \leq a$.

We now extend some characterizations of finite convex geometries to strongly coatomic,
spatial geometries.
In this setting we want to think of a lattice in terms of its standard
representation as a closure system on its set of join irreducibles.
Note that properties (1)--(2) of the next theorem are about closure systems, 
while (3) and (5)--(6) are lattice properties; (4) is more geometric in nature.
For the finite case, various parts of the 
theorem are due to S.P.~Avann, P.~Edelman and R.~Jamison, and V.~Duquenne.

\begin{thm}\label{SCSgeom}
Let $L$ be an strongly coatomic, spatial lattice. Then the following are equivalent.
\begin{enumerate}
\item $L$ is the closure lattice $\op{Cld} (X, \gamma)$ of a closure
system $(X,\gamma)$ with the AEP.
\item $L$ is the closure lattice $\op{Cld} (X, \gamma)$ of a closure
system $(X,\gamma)$ with the property that 
if $A\prec B$ in $\op{Cld} (X, \gamma)$, then $|B\setminus A|=1$.
%%%\item $L$ is locally distributive and lower semimodular.
\item $L$ is join semidistributive and lower semimodular.
\item Every element $w \in L$ is the join of\/ $\op{Ex}(w)$.  
\item Every element $w \in L$ has a unique irredundant join decomposition
into join irreducible elements, which is canonical, i.e., it refines every
other join representation of $w$.
\item For every pair of elements $w$, $c \in L$ with 
$c \prec w$, there exists a unique $j$ in $\op{Ji}(L)$
such that $j \leq w$ and $j \nleq c$.
\end{enumerate}
\end{thm}

Note that we need only assume that $L$ is spatial, but that strong spatiality
will follow, for example from (3) and Lemma~\ref{lsmtoss}.

\begin{proof}
The equivalence of (1) and (2) is Theorem \ref{covCG}.

To see that (2) $\Rightarrow$ (3), consider a closure system $(X,\gamma)$ 
satisfying (2).  Clearly (2) implies that $\op{Cld}(X,\gamma)$ is 
lower semidmodular; we want to show that it is join semidistributive.
This can be done by proving that every closed set $B$ has a canonical join
decomposition in $\op{Cld}(X,\gamma)$.  (Note that (3) $\Rightarrow$ (2) 
$\Rightarrow$ (1) of Theorem \ref{CharSD} holds in all complete lattices.)

For each $A \prec B$ in $\op{Cld}(X,\gamma)$, let $\{ x_A \} = B \setminus A$.
The claim is that $B = \Join_{A \prec B} \gamma(x_A)$ canonically.
Let $R$ denote the right hand side, and note that $B \supseteq R$.
If $B \supset R$ properly, then there would exist $C$ such that 
$R \leq C \prec B$.  That would imply $x_C \in R \subseteq C$, a contradiction.
Thus $B=R$.  Moreover, the join is irredundant as $x_A \in C$ for any pair
of distinct lower covers $A$, $C \prec B$.

Now suppose $B = \Join_{i \in I} D_i$ for some closed sets $D_i$ in $\op{Cld}(X,\gamma)$.
For each $A \prec B$, we have $\bigcup D_i \not\subseteq A$, so that there exists
an $i_0$ with $D_{i_0} \not\subseteq A$.  Since $D_{i_0} \subseteq B$, this implies
$x_A \in D_{i_0}$, whence $\gamma(x_A) \leq D_{i_0}$. 
We have shown that $\{ \gamma(x_A) : A \prec B \} \ll \{ D_i : i \in I \}$,
as required for a canonical join decomposition.
This proves (3).

Properties (4)--(6) are variations on a theme.  Let us work with (6) for now,
returning to their equivalence later.

Assume (3), that $L$ is join semidistributive and lower semimodular,
and we want to show that (6) holds.
Let $w \succ c$ in $L$.  The assumption that $L$ is spatial means that there 
is a completely join irreducible element $j$ such that $j \leq w$ but $j \nleq c$.  
By lower semimodularity, $j \succ j \meet c$, i.e., $j_* \leq c$,
so $j$ is minimal with this property.
By $\op{SD}_\join$, there can be only one such minimal element.  Thus
$j$ is the unique minimal element with $j \leq w$, $j \nleq c$.  

Note that (6) implies (2) immediately for the standard representation of a 
apatial lattice as a closure system on its set of join irreducible elements.
Indeed, (6) says that if $c \prec w$ in $L$, then 
$\op{Ji}(w) \setminus \op{Ji}(c) = \{ j \}$.

We turn to the equivalence of conditions (4)--(6).  
Let $w \in L$, and assume (6), that for each covering pair $c \prec w$ 
there is a unique $j_c$ in $\op{Ji}(L)$ such that $j_c \leq w$ and $j_c \nleq c$.
Then $\Join \{ j_c: c \prec w \} = w$, since the join is below $w$
but not below any of its lower covers.  On the other hand, each such
$j_c$ is extreme, as $c = \Join (\op{Ji}(w) \setminus \{ j_c \})$,
because $j_c$ is the only join irreducible below $w$ that not below $c$.
Thus (4) holds.

Next assume (4), and let us show that $w = \Join \op{Ex}(w)$ is the unique
irredundant decomposition of an element $w \in L$ into join irreducibles,
and that this decomposition refines any other decomposition.
Suppose $w = \Join B$ is another such decomposition.  Because each 
$j \in \op{Ex}(w)$ is completely join prime in the ideal $\op{id}(w)$,
we have $j \leq b$ for some $b \in B$.  As there are no other join irreducibles
in the interval $[j,w]$, this implies $j=b$.  Thus $\op{Ex}(w) \subseteq B$,
and since both are irredundant decompositions, $\op{Ex}(w) = B$.
This proves (5).

Finally, assume (5), that each element of $L$ has a unique irredundant join 
decomposition into join irreducible elements, which is canonical.
Let $w \in L$ with $w = \Join K$ its canonical decomposition, and let $c \prec w$.
There exists a $k_0 \in K$ such that $k_0 \nleq c$.  
We want to show that $k_0$ is the only join irreducible element of $L$ that is
below $w$ but not below $c$.

Note that for any element $p \leq w$, we have 
$p \nleq c$ if and only if $k_0 \leq p$.
Clearly $k_0 \leq p$ implies $p \nleq c$.  Conversely, if $p \nleq c$ then
$w = p \join c$, and since $w=\Join K$ canonically, $K \ll \{ p,c \}$.  
In particular, $k_0 \leq p$.
This also implies that $k \leq c$ for all $k \in K \setminus \{ k_0 \}$.

Now let $p \in \op{Ji}(L)$ with $k_0 \leq p \leq w$.  
Let $K' = \{ k \in K : k \nleq p \}$.  
We claim that $w = p \join \Join K'$ is an irredundant decomposition of $w$
into join irreducibles.
As noted above, $\Join K' \leq c$.  Consider any $k_1 \in K'$.  There is a lower
cover $d \prec w$ such that $k_1 \nleq d$.  By the argument above applied to $k_1$
we have $p \join \Join (K' \setminus \{ k_1 \}) \leq d$.  Thus the decomposition
is irredundant.  As $w = \Join K$ is the unique such decomposition, it must be that
$p=k_0$.  Thus (5) holds, completing the proof.

(When $L$ is finite, it is not necessary to assume that the decomposition
in (6) is canonical.  This will follow from uniqueness, since every join
representation of an element refines to an irredundant one in a finite lattice.)
\end{proof}

Conspicuously missing from the list of equivalences of Theorem~\ref{SCSgeom}
is an analogue of Dilworth's characterization of finite convex geometries
as lattices that are locally distributive.
A strongly coatomic complete lattice is said to be \emph{locally distributive} 
(or \emph{lower locally distributive} or \emph{meet distributive}) 
if for any $x \in L$ the interval $[\mu(x), x]$, where
$\mu(x)= \bigwedge \{ y :\, y \prec x \}$, is a distributive lattice 
(and hence for finite lattices a boolean algebra). 

The property \emph{$L$ is locally distributive} is missing from the list because
it is not equivalent in this setting.
Consider the lattice $(\omega +1)^d \times \mathbf 2$, with its atom doubled. 
This lattice is strongly coatomic, strongly spatial, locally distributive,
and even join semidistributive, but it is not lower semimodular.

On the other hand, it is straightforward to show that property (2) of
Theorem~\ref{SCSgeom} implies local distributivity in any closure system.
One solution is to further restrict the class of lattices under consideration.
For lattices that are also lower continuous, M. Semenova has shown that local
distributivity is also equivalent.

\begin{thm}\label{T:CharCS} \cite{MS00}
Let $L$ be a strongly coatomic, lower continuous, complete lattice. 
Then the conditions of Theorem \ref{SCSgeom} are equivalent to
\begin{enumerate}
\item[(7)] $L$ is locally distributive.
\end{enumerate}
\end{thm}

However, it may be that we need only add the hypothesis of lower semimodularity
to complete the equivalences.

\begin{prob}
Does every strongly coatomic, locally distributive, lower semimodular,
complete lattice satisfy the properties of Theorem~\ref{SCSgeom}?
\end{prob}

It would suffice, for example, to show join semidistributivity.
But so far, a solution eludes us.

Atomistic convex geometries were characterized in Proposition~3.1 of Adaricheva,
Gorbunov and Tumanov \cite{AGT03}.  For the strongly coatomic case, the proof is 
particularly easy.

% Prop 3.1. in [3] is stronger, since it requires only complete, atomistic and SD
%for atoms. Moreover, it is iff statement. KA
\begin{cor}\label{C:JSDfCG}
Any atomistic, strongly coatomic, join semidistributive
lattice is the closure lattice of some convex geometry.
\end{cor}
\begin{proof} 
Let $L$ be a lattice with these properties, and let $c \prec w$ in $L$.
By atomicity, there is an atom $t$ such that $t \leq w$ and $t \nleq c$.
There cannot be two distinct such atoms, say $t_1$ and $t_2$, else
$c \join t_1 = c \join t_2 = w > c = c \join (t_1 \meet t_2)$, 
contradicting join semidistributivity.  
Hence $\op{Ji}(w) \setminus \op{Ji}(c) = \{ t \}$, so that condition~(2) 
of Theorem~\ref{SCSgeom} holds.
%It is enough to show that a strongly coatomic, lower continuous atomistic join semidistributive lattice 
%is lower semimodular.  Indeed, it follows from join semidistributivity that if $a\prec b$, 
%then there exists a unique atom $t$ such that $t \leq b$, $t \nleq a$.
%For any element $c \leq b$, either $t\leq c$ and thus $c$ and $c\wedge a$ differ by the single atom $t$, 
%so that $c\wedge a \prec c$, or else $t\nleq c$, in which case every atom below $c$ is below $a$, whence $c\leq a$.
%Thus $a \prec b$ implies $c \wedge a \preceq c \wedge b$, which is lower semimodularity.  
\end{proof}

\section{Examples of strongly coatomic convex geometries}

There are natural examples of the kind of geometries described in 
Corollary \ref{C:JSDfCG}, obtained by taking standard convex geometries and 
adding finiteness conditions to ensure strong coatomicity.  
Recall that every chain in an ordered set $P$ is finite if and only if $P$ 
satisfies both the ACC and DCC.

\begin{thm}
\emph{(1)}   If $P$ is an ordered set in which every chain is finite,
then the lattice of convex subsets $\op{Co}(P)$ is an atomistic, strongly
coatomic convex geometry.

\emph{(2)}  If $S$ is a meet semilatiice that satisfies the ACC,
then the lattice of subsemilattices $\op{Sub}_\meet (S)$ is an atomistic, strongly
coatomic convex geometry.
\end{thm}

Note that in each example, the closure operator is algebraic. 

\begin{proof}
(1)  A subset $A$ of an ordered set $P$ is \emph{convex} if 
$a_1$, $a_2 \in A$ and $a_1 \leq x \leq a_2$ implies $x \in A$. 
We know that the convex hull operator on an ordered set satisfies
the AEP; it remains to show that if $P$ has the property that every
chain is finite, then $\op{Co}(P)$ is strongly coatomic.  

Suppose that $P$ has that property, and that $A<B$ in $\op{Co}(P)$.
Let $b_0 \in B \setminus A$.  Then either
$\op{id}(b_0) \cap A = \emptyset$ or $\op{fil}(b_0) \cap A = \emptyset$,
w.l.o.g.~the former.  Choose $b_1$ minimal in $B \cap \op{id}(b_0)$.
Then $A \subseteq B \setminus \{ b_1 \} \prec B$ in $\op{Co}(P)$.
Thus $\op{Co}(P)$ is strongly coatomic.

(2) Now consider $\op{Sub}_{\meet}(S)$ for a meet semilattice $S$.
Again, we know that the subsemilattice operator satisfies the AEP, and
it remains to show that if $S$ satisfies the ACC,
then $\op{Sub}_{\meet}(S)$ is strongly coatomic.  
But this is easy:
if $A < B$ in $\op{Sub}_{\meet}(S)$, then we can choose $b_0$ maximal in $B \setminus A$
to obtain $A \leq B \setminus \{ b_0 \} \prec B$.
\end{proof}

We can even combine these examples:  if $S$ is a meet semilattice in which every 
principal filter $\op{fil}(x)$ is a finite tree,
then the lattice of convex subsemilattices of $S$ is a strongly coatomic, lower
continuous convex geometry.  See Adaricheva \cite{KA88} and Cheong and Jones \cite{CJ}.

In a similar vein, if $(P,\leq)$ is an ordered set in which every chain is finite 
and every interval is finite, then the lattice of suborders of $\leq$ on $P$ is a 
strongly coatomic and lower continuous convex geometry.  See Semenova \cite{MS99}.

Another construction yields strongly coatomic, lower continuous convex geometries that need not be atomistic.  
Our inspiration is the fact that a geometric lattice is isomorphic to the ideal 
lattice of its finite dimensional elements.  (There is no chance for a similar 
characterization here, since for any non-limit ordinal $\alpha$, the dual $\alpha^d$ 
is a strongly coatomic, lower continuous convex geometry.)  
Our construction uses J\'onsson and Rival's 
characterization of join semidistributive varieties \cite{JR79}.  

Define certain lattice terms recursively:  for $k \geq 0$, 
\begin{alignat*}{2}
y_0 &= y &  z_0 &= z \\
y_{k+1} &= y \meet (x \join z_k)\qquad  & z_{k+1} &= z \meet (x \join y_k) .
\end{alignat*}
Then consider the lattice inclusions
\[ \op{SD}_{\join}(k) \qquad y_k \leq x \join (y \meet z) \ . \]
These are equivalent to the corresponding identities $x \join y_k \approx x \join (y \meet z)$.
For example, $\op{SD}_{\join}(1)$ is equivalent to the distributive law.

\begin{lm} \cite{JR79}
Let $\mathcal V$ be a lattice variety.  Then every lattice in $\mathcal V$ is join semidistributive if and only if\/ $\mathcal V$ satisfies $\op{SD}_{\join}(n)$ for some $n < \omega$.
\end{lm} 

Let $\mathbf{SD}_\join(n)$ be the variety of all lattices satisfying $\op{SD}_\join(n)$.

\begin{thm}
Let $L_0$ be a lattice with the following properties.
\begin{itemize}
\item $\op{fil}(x)$ is finite for each $x \in L_0$. 
\item $L_0 \in \mathbf{SD}_{\join}(n)$ for some $n < \omega$. 
\item $L_0$ is lower semimodular.
\end{itemize}
Then the filter lattice $L = \op{Fil}(L_0)$ is an strongly coatomic, lower continuous convex geometry.
\end{thm}

%%The simplest example would be when $L_0$ is a distributive lattice such that
%%$\op{fil}(x)$ is finite for each $x \in L_0$. 

\begin{proof}
As usual, we order the filter lattice by reverse set inclusion:  $F \leq G$ iff $F \supseteq G$.
The filter lattice of any lattice is lower continuous and satisfies the
equations of the original, in particular $\op{SD}_{\join}(n)$ in this case. 
It remains to show that $L$ is strongly coatomic and lower semimodular.

Suppose $F < G$ in $L$, i.e., $F \supset G$.  Let $k$ be an element maximal
in $F \setminus G$, and note that $k$ is meet irreducible.  We claim that
the filter generated by $G \cup \{ k \}$, say $H = \op{fil}(G,k)$, satisfies $F \leq H \prec G$. 
Let $\ell$ be any element
of $H \setminus G$.  Then $\ell \geq g \meet k$ for some $g \in G$, and we
may take $g \leq k^*$, where $k^*$ denotes the unique upper cover of $k$ in $L_0$.  
In that case, by lower semimodularity, 
$g \succ g \meet k$, whence also $g \meet k = g \meet \ell$.  
It follows that $H = \op{fil}(G,\ell)$, and since $\ell$ is arbitrary, $H \prec G$. 
Thus $L$ is strongly coatomic.

The proof that $L$ is lower semimodular is an adaptation
of that for the corresponding dual claim in Theorem 11.1 of \cite{LTN}.  
Assume that $L$ is lower semimodular, and suppose that 
$F \prec F \join G = F \cap G$ in $\op{Fil}(L_0)$.
Choose an element $a$ maximal in $F \setminus G$, and note that
$a$ is meet irreducible, thus by the finiteness of $\op{fil}(a)$
completely meet irreducible.  
Then $F = (F \join G) \,\meet\, \op{fil}(a)$, and hence 
$F \meet G = \op{fil}(a) \,\meet\, G$.  
Let $x$ be any element in $(F \meet G) \setminus G$.  
Since $x \in F \meet G$, there exists $g \in G$ such that 
$x \geq a \meet g$.  
Because $L$ is lower semimodular, $a \meet g \prec a^* \meet g$.  
On the other hand, every element of $L$ is a meet of finitely many 
meet irreducibles, so $x \notin G$ implies there exists a meet
irreducible element $b \geq x$ with $b \notin G$. 
Now $b \geq a \meet g$ and $b \ngeq g$, so $b \meet g = a \meet g$, 
whence $a \geq b \meet g$.  Thus 
$\op{fil}(b) \,\meet\, G = \op{fil}(a) \,\meet\, G = F \meet G$; 
if follows \emph{a fortiori} that $\op{fil}(x) \,\meet\, G = F \meet G$.
As this holds for every $x \in (F \meet G) \setminus G$, we have 
$F \meet G \prec G$, as desired.
\end{proof}

So in particular, we could take $L_0$ to be the elements of finite depth
in an infinite direct product of finite convex geometries that satisfy 
$\op{SD}_{\join}(n)$ for some fixed $n$.

The examples so far have all been algebraic closure operators.  For a non-algebraic example of an strongly coatomic, lower continuous convex geometry,
we form a closure system $( \omega, \gamma)$ on the natural numbers $\omega$.  
Define a subset $S \subseteq \omega$ to be $\gamma$-closed if either $0 \in S$ or $S$ is finite.
Clearly the closed sets are closed under arbitrary intersections, so $\op{Cld}(\omega,\gamma)$ is a complete lattice.
Moreover, the lower covers of a nonempty closed set $S$ 
are all sets $S \setminus \{ x \}$ with $x \in S$ if $S$ is finite, 
and all sets $S \setminus \{ x \}$ with $0 \ne x \in S$ if $S$ is infinite.
It follows easily that $\op{Cld}(\omega,\gamma)$ is strongly coatomic, and it has the property that $T \prec S$ implies
$|S \setminus T|=1$ of Theorem~\ref{T:CharCS}.  
Note that $\op{Cld}(\omega,\gamma)$ is a sublattice of the subset lattice $\op{Pow}(\omega)$, closed under arbitrary
intersections and finite unions (but not infinite unions).  This clearly makes $\op{Cld}(\omega, \gamma)$ lower continuous.
To see that it is not algebraic, we show that it is not upper continuous.  For $k \ge 1$, let $F_k = \{ 1, \dots, k \}$.
This is a chain with $\Join F_k = \omega$, and hence
\[  \{ 0 \} = \{ 0 \} \meet \Join F_k \supset \Join ( \{ 0 \} \meet F_k ) = \emptyset  .\] 
Thus $\op{Cld}(\omega,\gamma)$ is a non-algebraic strongly coatomic, lower continuous convex geometry.

\section{Discussion}

In some sense, algebraic closure operators are the natural settings for any type
of geometry.  On the other hand, Crawley and Dilworth's setting of dually algebraic
and strongly coatomic gives the nice equivalence of local distributivity and
unique representations.  Since dually algebraic lattices are
lower continuous, the hypothesis of strongly coatomic, lower continuous is slightly 
weaker.   Indeed, most of the equivalences for convex geometries
remain valid for lattices that are strongly coatomic and spatial.
The question remains whether local distributivity and lower semimodularity are equivalent
to the other characterizations  
under the weaker hypothesis of a strongly coatomic, spatial complete lattice.   
Thus the question can be phrased:  \emph{Does every strongly coatomic, spatial,
locally distributive and lower semimodular closure system satisfy the anti-exchange
property?}

Strong coatomicity is a strong assumption.  Some progress has been made towards 
relaxing this to the weakly atomic property (every interval contains a cover) in
\cite{AN15} and \cite{AP11}, but more could doubtless be done along these lines.

\medskip

The authors would like to thank the referee for prompting us to consider the 
general class of strongly coatomic and spatial lattices,
and other helpful suggestions.
%%uuuuuuuuuuuuuuuuuuuuuuuuuuuuuuuuuuuuuuuuuuuuuu
%%vvvvvvvvvvvvvvvvvvvvvvvvvvvvvvvvvvvvvvvv

%%%%%%%%%%%%%%%%%%%%%%%
\end{document}